\documentclass[12pt,reqno,a4paper]{amsart}
\usepackage{amsfonts}
\usepackage{amsmath}

\newtheorem{theorem}{Theorem}[section]
\newtheorem{corollary}[theorem]{Corollary}
\newtheorem{lemma}[theorem]{Lemma}
\theoremstyle{definition}
\newtheorem{definition}[theorem]{Definition}
\theoremstyle{remark}
\newtheorem{remark}[theorem]{Remark}
\input tcilatex

\begin{document}
\title[On subclasses of uniformly spirallike functions ...]{On subclasses of
uniformly spirallike functions associated with generalized Bessel functions}
\author{B.A. Frasin}
\address{Faculty of Science, Department of Mathematics, Al al-Bayt
University, Mafraq, Jordan}
\email{bafrasin@yahoo.com}
\author{Ibtisam Aldawish\textbf{\ }}
\address{ Department of Mathematics and Statistics, College of Science,
lMSIU (Imam Mohammed Ibn Saud Islamic University), P.O.Box 90950, Riyadh
11623, Saudi Arabia.}
\email{imaldawish@imamu.edu.sa}

\begin{abstract}
The main object of this paper is to find necessary and sufficient conditions
for generalized Bessel functions of first kind $zu_{p}(z)$ to be in the
classes $\mathcal{SP}_{p}(\alpha ,\beta )$ and $\mathcal{UCSP}(\alpha ,\beta
)$ of uniformly spirallike functions and also give necessary and sufficient
conditions for $z(2-u_{p}(z))$ to be in the above classes. Furthermore, we
give necessary and sufficient conditions for $\mathcal{I}(\kappa ,c)f$ \ to
be in $\mathcal{UCSPT}(\alpha ,\beta )$ provided that the function $f$ is in
the class $\mathcal{R}^{\tau }(A,B)$. Finally, we give conditions for the
integral operator $\mathcal{G(}\kappa ,c,z\mathcal{)=}%
\int_{0}^{z}(2-u_{p}(t))dt$ to be in the class $\mathcal{UCSPT}(\alpha
,\beta ).$ Several corollaries and consequences of the main results are also
considered.


\end{abstract}

\maketitle

\section{\protect\bigskip Introduction and definitions}

Let $\mathcal{A}$ denote the class of the normalized functions of the form%
\begin{equation}
f(z)=z+\sum_{n=2}^{\infty }a_{n}z^{n},  \label{f}
\end{equation}%
which are analytic in the open unit disk $\mathbb{U}=\{z\in \mathbb{C}%
:\left\vert z\right\vert <1\}.$ Further, let $\mathcal{T}$ \ be a subclass
of $\mathcal{A}$ consisting of functions of the form, 
\begin{equation}
f(z)=z-\sum\limits_{n=2}^{\infty }\left\vert a_{n}\right\vert z^{n},\qquad
z\in \mathbb{U}\text{.}  \label{m1}
\end{equation}%
A function $f\in \mathcal{A}$ is spirallike if 
\begin{equation*}
\mathfrak{R}\left( e^{-i\alpha }\frac{zf^{\prime }(z)}{f(z)}\right) >0,\ 
\end{equation*}%
for some $\alpha $ with $\left\vert \alpha \right\vert <\pi /2\ $and for all 
$z\in \mathbb{U}$ . Also $f(z)$ is convex spirallike if $zf^{\prime }(z)$ is
spirallike.

In \cite{sel}, Selvaraj and Geetha introduced the following subclasses of
unifromly spirallike and convex spirallike functions.

\begin{definition}
A function $f$ of the form (\ref{f}) is said to be in the class $\mathcal{SP}%
_{p}(\alpha ,\beta )$ if it satisfies the following condition:%
\begin{equation*}
\mathfrak{R}\left\{ e^{-i\alpha }\left( \frac{zf^{\prime }(z)}{f(z)}\right)
\right\} >\left\vert \frac{zf^{\prime }(z)}{f^{\prime }(z)}-1\right\vert
+\beta ~\ \ \ \ \ (z\in \mathbb{U};\left\vert \alpha \right\vert <\pi
/2~;0\leq \beta <1)
\end{equation*}
\end{definition}

\textit{and }$f\in $\textit{\ }$\mathcal{UCSP}(\alpha ,\beta )$\textit{\ if
and only if }$zf^{\prime }(z)\in \mathcal{SP}_{p}(\alpha ,\beta ).$

We write%
\begin{equation*}
\mathcal{SP}_{p}\mathcal{T}(\alpha ,\beta )=\mathcal{SP}_{p}(\alpha ,\beta
)\cap \mathcal{T}
\end{equation*}%
and 
\begin{equation*}
\mathcal{UCSPT}(\alpha ,\beta )=\mathcal{UCSP}(\alpha ,\beta )\cap \mathcal{T%
}\text{.}
\end{equation*}

In particular, we note that $\mathcal{SP}_{p}(\alpha ,0)=\mathcal{SP}%
_{p}(\alpha )$\ and $\mathcal{UCSP}(\alpha ,0)=\mathcal{UCSP}(\alpha ),$ the
classes of uniformly spirallike and uniformly convex spirallike were
introduced by Ravichandran et al. \cite{rav}. \bigskip For $\alpha =0,$ the
classes $\mathcal{UCSP}(\alpha )$ and $\mathcal{SP}_{p}(\alpha ),$
respectively reduces to the classes $\mathcal{UCV}$ and $\mathcal{SP}$
introduced and studied by Ronning \cite{ron2}.

For more interesting developments of some related subclasses of uniformly
spirallike and uniformly convex spirallike, the readers may be referred to
the works of Frasin \cite{fr,fr3}, Goodman \cite{goo1,goo2}, Tariq Al-Hawary
and Frasin \cite{fr2}, Kanas and Wisniowska \cite{kan1, kan2} and Ronning 
\cite{ron1, ron2}.

A function $f\in \mathcal{A}$ is said to be in the class $\mathcal{R}^{\tau
}(A,B)$,$\tau \in \mathbb{C}\backslash \{0\}$, $-1\leq B<A\leq 1,$ if it
satisfies the inequality%
\begin{equation*}
\left\vert \frac{f^{\prime }(z)-1}{(A-B)\tau -B[f^{\prime }(z)-1]}%
\right\vert <1,\ \ \ \;z\in \mathbb{U}.
\end{equation*}

\bigskip This class was introduced by Dixit and Pal \cite{dix}.

\noindent \bigskip The generalized Bessel function $w_{p}$ (see, \cite{bar}%
)is defined as a particular solution of the linear differential equation%
\begin{equation*}
zw^{\prime \prime }(z)+bzw^{\prime }(z)+[cz^{2}-p^{2}+(1-b)p]w(z)=0,
\end{equation*}%
where $b,p,c\in \mathbb{C}$. The analytic function $w_{p}$ has the form 
\begin{equation*}
w_{p}(z)=\sum\limits_{n=0}^{\infty }\frac{(-1)^{n}(c)^{n}}{n!\Gamma (p+n+%
\frac{b+1}{2})}.\left( \frac{z}{2}\right) ^{2n+p},\ \ \ \ \ \ \ z\in \mathbb{%
C}\text{.}
\end{equation*}%
Now, the generalized and normalized Bessel function $u_{p\text{ }}$is
defined with the transformation 
\begin{eqnarray*}
u_{p}(z) &=&2^{p}\Gamma (p+n+\frac{b+1}{2})z^{-p/2}w_{p}(z^{1/2}) \\
&=&\sum\limits_{n=0}^{\infty }\frac{(-c/4)^{n}}{(\kappa )_{n}n!}z^{n},
\end{eqnarray*}%
where \bigskip $\kappa =p+(b+1)/2\neq 0,-1,-2,\ldots $ and $(a)_{n}$ is the
well-known Pochhammer (or Appell) symbol, defined in terms of the Euler
Gamma function for $a\neq 0,-1,-2,\ldots $ by

\begin{equation*}
(a)_{n}=\frac{\Gamma (a+n)}{\Gamma (a)}=\left\{ 
\begin{array}{c}
1,\ \ \ \ \ \ \ \ \ \ \ \ \ \ \ \ \ \ \ \ \ \ \ \ \ \ \ \ \ \ \ \ \ \ \ \ \
\ \ \ \ \ \ \text{if }n=0\ \  \\ 
\ \ \  \\ 
a(a+1)(a+2)\ldots (a+n-1),\ \ \ \ \ \text{if }n\in \mathbb{N}\text{.}\ 
\end{array}%
\right.
\end{equation*}%
The function $u_{p}$ is analytic on $\mathbb{C}$ and satisfies the
second-order linear differential equation%
\begin{equation*}
4z^{2}u^{\prime \prime }(z)+2(2p+b+1)zu^{\prime }(z)+czu(z)=0.
\end{equation*}%
Using the Hadamard product, we now considered a linear operator $\mathcal{I}%
(\kappa ,c):\mathcal{A\rightarrow A}$ defined by%
\begin{equation*}
\mathcal{I}(\kappa ,c)f=zu_{p}(z)\ast f(z)=z+\sum\limits_{n=2}^{\infty }%
\frac{(-c/4)^{n-1}}{(\kappa )_{n-1}(n-1)!}a_{n}z^{n},\text{ }
\end{equation*}%
where $\ast $ denote the convolution or Hadamard product of two series.

\bigskip The study of the generalized Bessel function is a recent
interesting topic in geometric function theory. We refer, in this
connection, to the works of \cite{bar, bar2, bar4} and others.

Motivated by results on connections between various subclasses of analytic
univalent functions by using hypergeometric functions (see, for example, 
\cite{1,fr4,2,8,9})), and the work done in \cite{t,m,m2,p}, we determine
necessary and sufficient conditions for $zu_{p}(z)$ to be in $\mathcal{SP}%
_{p}(\alpha ,\beta )$ and $\mathcal{UCSP}(\alpha ,\beta )$ and also give
necessary and sufficient conditions for $z(2-u_{p}(z))$ to be in the
function classes $\mathcal{SP}_{p}\mathcal{T}(\alpha ,\beta )$ and $\mathcal{%
UCSPT}(\alpha ,\beta )$. Furthermore, we give necessary and sufficient
conditions for $\mathcal{I}(\kappa ,c)f$ \ to be in $\mathcal{UCSPT}(\alpha
,\beta )$ provided that the function $f$ is in the class $\mathcal{R}^{\tau
}(A,B)$. Finally, we give conditions for the integral operator $\mathcal{G(}%
\kappa ,c,z\mathcal{)=}\int_{0}^{z}(2-u_{p}(t))dt$ to be in the class $%
\mathcal{UCSPT}(\alpha ,\beta ).$\bigskip

To establish our main results, we need the following Lemmas.

\begin{lemma}
(see \cite{sel}) (i) A sufficient condition for a function\ $f$\ of the form
(\ref{f}) to be in the class $\mathcal{SP}_{p}(\alpha ,\beta )$ is that%
\begin{equation}
\sum\limits_{n=2}^{\infty }(2n-\cos \alpha -\beta )\left\vert
a_{n}\right\vert \leq \cos \alpha -\beta ~~\ \ \ \ \ (\left\vert \alpha
\right\vert <\pi /2~;0\leq \beta <1)  \label{t1}
\end{equation}
\end{lemma}

\textit{and a necessary and sufficient condition for a function }$f$\textit{%
\ of the form (\ref{m1}) to be in the class }$\mathcal{SP}_{p}\mathcal{T}%
(\alpha ,\beta )$\textit{\ is that the condition (\ref{t1}) is satisfied. In
particular, when }$\beta =0,$ \textit{we obtain a sufficient condition for a
function\ }$f$\textit{\ of the form (\ref{f}) to be in the class }$\mathcal{%
SP}_{p}(\alpha )$\textit{\ is that}%
\begin{equation}
\sum\limits_{n=2}^{\infty }(2n-\cos \alpha )\left\vert a_{n}\right\vert \leq
\cos \alpha ~~\ \ \ \ \ (\left\vert \alpha \right\vert <\pi /2)  \label{t2}
\end{equation}

\textit{and a necessary and sufficient condition for a function }$f$\textit{%
\ of the form (\ref{m1}) to be in the class }$\mathcal{SP}_{p}\mathcal{T}%
(\alpha )$\textit{\ is that the condition (\ref{t2}) is satisfied.}

\textit{\ (ii) A sufficient condition for a function\ }$f$\textit{\ of the
form (\ref{f}) to be in the class }$\mathcal{UCSP}(\alpha ,\beta )$\textit{\
is that}%
\begin{equation}
\sum\limits_{n=2}^{\infty }n(2n-\cos \alpha -\beta )\left\vert
a_{n}\right\vert \leq \cos \alpha -\beta ~~\ \ \ \ \ (\left\vert \alpha
\right\vert <\pi /2~;0\leq \beta <1)  \label{b1}
\end{equation}

\textit{and a necessary and sufficient condition for a function }$f$\textit{%
\ of the form (\ref{m1}) to be in the class }$\mathcal{UCSPT}(\alpha ,\beta
) $\textit{\ is that the condition (\ref{b1}) is satisfied. In particular,
when }$\beta =0,$ \textit{we obtain a sufficient condition for a function\ }$%
f$\textit{\ of the form (\ref{f}) to be in the class }$\mathcal{UCSP}(\alpha
)$\textit{\ is that}%
\begin{equation}
\sum\limits_{n=2}^{\infty }n(2n-\cos \alpha )\left\vert a_{n}\right\vert
\leq \cos \alpha ~~\ \ \ \ \ (\left\vert \alpha \right\vert <\pi /2)
\label{b2}
\end{equation}%
\textit{and a necessary and sufficient condition for a function }$f$\textit{%
\ of the form (\ref{m1}) to be in the class }$\mathcal{UCSPT}(\alpha )$%
\textit{\ is that the condition (\ref{b2}) is satisfied.}

\begin{lemma}
\label{lem3}\cite{dix}If $f$\textit{\ }$\in $\textit{\ }$\mathcal{R}^{\tau
}(A,B)$ is of the form (\ref{f}), then%
\begin{equation*}
\left\vert a_{n}\right\vert \leq (A-B)\frac{\left\vert \tau \right\vert }{n}%
,\ \ \ \ \ \ n\in \mathbb{N}-\{1\}\text{.}
\end{equation*}
\end{lemma}

\textit{The result is sharp for the function }%
\begin{equation*}
f(z)=\int_{0}^{z}(1+(A-B)\frac{\tau t^{n-1}}{1+Bt^{n-1}})dt,\ \ \ \ \ \
(z\in \mathbb{U};n\in \mathbb{N}-\{1\})\text{.}
\end{equation*}

\begin{lemma}
(see \cite{bar4})If $b,p,c\in \mathbb{C}$ and $\ \kappa \neq 0,-1,-2,\ldots $
then the function $u_{p}$ satisfies the recursive relations 
\begin{equation*}
u_{p}^{\prime }(z)=\frac{(-c/4)}{\kappa }u_{p+1}(z)\text{ and}~\text{\ }%
u_{p}^{\prime \prime }(z)=\frac{(-c/4)^{2}}{\kappa (\kappa +1)}u_{p+2}(z),~
\end{equation*}%
for all $z\in \mathbb{C}$.
\end{lemma}

\section{The necessary and sufficient conditions}

Unless otherwise mentioned, we shall assume in this paper that $\left\vert
\alpha \right\vert <\pi /2~$and $0\leq \beta <1.$

First we obtain the necessary condition for $zu_{p}$ to be in $\mathcal{SP}%
_{p}(\alpha ,\beta ).$

\begin{theorem}
\label{th1}If $c<0,$ $\kappa >0(\kappa \neq 0,-1,-2,\ldots ),$then $zu_{p}$
is in $\mathcal{SP}_{p}(\alpha ,\beta )$ if 
\begin{equation}
2u_{p}^{\prime }(1)+(2-\cos \alpha -\beta )(u_{p}(1)-1)\leq \cos \alpha
-\beta .  \label{hh}
\end{equation}
\end{theorem}

\begin{proof}
Since%
\begin{equation}
zu_{p}(z)=z+\sum\limits_{n=2}^{\infty }\frac{(-c/4)^{n-1}}{(\kappa
)_{n-1}(n-1)!}z^{n},
\end{equation}%
according to (\ref{t1}), we must show that%
\begin{equation}
\sum\limits_{n=2}^{\infty }(2n-\cos \alpha -\beta )\frac{(-c/4)^{n-1}}{%
(\kappa )_{n-1}(n-1)!}\leq \cos \alpha -\beta .
\end{equation}%
Writing%
\begin{equation}
n=(n-1)+1,
\end{equation}%
we have%
\begin{eqnarray}
&&\sum\limits_{n=2}^{\infty }(2n-\cos \alpha -\beta )\frac{(-c/4)^{n-1}}{%
(\kappa )_{n-1}(n-1)!}  \notag \\
&=&2\sum\limits_{n=2}^{\infty }(n-1)\frac{(-c/4)^{n-1}}{(\kappa )_{n-1}(n-1)!%
}+\sum\limits_{n=2}^{\infty }(2-\cos \alpha -\beta )\frac{(-c/4)^{n-1}}{%
(\kappa )_{n-1}(n-1)!}  \label{u} \\
&=&2\sum\limits_{n=2}^{\infty }\frac{(-c/4)^{n-1}}{(\kappa )_{n-1}(n-2)!}%
+\sum\limits_{n=2}^{\infty }(2-\cos \alpha -\beta )\frac{(-c/4)^{n-1}}{%
(\kappa )_{n-1}(n-1)!}  \notag \\
&=&2\sum\limits_{n=0}^{\infty }\frac{(-c/4)^{n+1}}{(\kappa )_{n+1}n!}%
+(2-\cos \alpha -\beta )\sum\limits_{n=0}^{\infty }\frac{(-c/4)^{n+1}}{%
(\kappa )_{n+1}(n+1)!}  \notag \\
&=&\frac{2(-c/4)}{\kappa }\sum\limits_{n=0}^{\infty }\frac{(-c/4)^{n}}{%
(\kappa +1)_{n}n!}+(2-\cos \alpha -\beta )\sum\limits_{n=0}^{\infty }\frac{%
(-c/4)^{n+1}}{(\kappa )_{n+1}(n+1)!}  \notag \\
&=&\frac{2(-c/4)}{\kappa }u_{p+1}(1)+(2-\cos \alpha -\beta )(u_{p}(1)-1) 
\notag \\
&=&2u_{p}^{\prime }(1)+(2-\cos \alpha -\beta )(u_{p}(1)-1).  \notag
\end{eqnarray}%
But this last expression is bounded above by $\cos \alpha -\beta ~$if ~(\ref%
{hh}) holds.
\end{proof}

\begin{corollary}
\bigskip \label{cor1}If $c<0,$ $\kappa >0(\kappa \neq 0,-1,-2,\ldots ),$%
\textit{\ }then $z(2-u_{p}(z))$ is in $\mathcal{SP}_{p}\mathcal{T}(\alpha
,\beta )$ if and only if the condition (\ref{hh}) is satisfied.
\end{corollary}

\begin{proof}
Since%
\begin{equation}
z(2-u_{p}(z))=z-\sum\limits_{n=2}^{\infty }\frac{(-c/4)^{n-1}}{(\kappa
)_{n-1}(n-1)!}z^{n}.  \label{n5}
\end{equation}

By using the same techniques given in the proof of Theorem \ref{th1}, we
have Corollary \ref{cor1}.
\end{proof}

\begin{theorem}
\label{th66}If $c<0,$ $\kappa >0(\kappa \neq 0,-1,-2,\ldots ),$ then $zu_{p}$
is in $\mathcal{SP}_{p}(\alpha ,\beta )$ if%
\begin{equation}
e^{(\frac{-c}{4\kappa })}[\frac{-c}{2\kappa }+(2-\cos \alpha -\beta )(1-e^{(%
\frac{c}{4\kappa })})]\leq \cos \alpha -\beta .  \label{q}
\end{equation}
\end{theorem}

\begin{proof}
We note that $(\kappa )_{n-1}=\kappa (\kappa +1)(\kappa +2)\cdots (\kappa
+n-2)\geq \kappa (\kappa +1)^{n-2}\geq \kappa ^{n-1},\ \ \ \ (n\in \mathbb{N}%
).$ From (\ref{u}), we get 
\begin{eqnarray*}
&&\sum\limits_{n=2}^{\infty }(2n-\cos \alpha -\beta )\frac{(-c/4)^{n-1}}{%
(\kappa )_{n-1}(n-1)!} \\
&\leq &2\sum\limits_{n=2}^{\infty }(n-1)\frac{(-c/4\kappa )^{n-1}}{(n-1)!}%
+(2-\cos \alpha -\beta )\sum\limits_{n=2}^{\infty }\frac{(-c/4\kappa )^{n-1}%
}{(n-1)!} \\
&=&(-c/2\kappa )e^{-c/4\kappa }+(2-\cos \alpha -\beta )(e^{-c/4\kappa }-1).
\end{eqnarray*}

Therefore, we see that the last expression is bounded above by $\cos \alpha
-\beta ~$if (\ref{q}) is satisfied.
\end{proof}

\begin{corollary}
If $c<0,$ $\kappa >0(\kappa \neq 0,-1,-2,\ldots ),$\textit{\ }then $%
z(2-u_{p}(z))$ is in $\mathcal{SP}_{p}\mathcal{T}(\alpha ,\beta )$ if and
only if the condition (\ref{q}) is satisfied.
\end{corollary}

\begin{theorem}
\label{th22}If $c<0,$ $\kappa >0(\kappa \neq 0,-1,-2,\ldots ),$ then $%
zu_{p}(z)$ is in $\mathcal{UCSP}(\alpha ,\beta )$ if 
\begin{equation}
2u_{p}^{\prime \prime }(1)+(6-\cos \alpha -\beta )u_{p}^{\prime }(1)+(2-\cos
\alpha -\beta )(u_{p}(1)-1)\leq \cos \alpha -\beta .  \label{gh}
\end{equation}
\end{theorem}

\begin{proof}
In view of (\ref{b1}) , we must show that%
\begin{equation}
\sum\limits_{n=2}^{\infty }n(2n-\cos \alpha -\beta )\frac{(-c/4)^{n-1}}{%
(\kappa )_{n-1}(n-1)!}\leq \cos \alpha -\beta .
\end{equation}

Writing%
\begin{equation}
n=(n-1)+1,
\end{equation}%
and 
\begin{equation*}
n^{2}=(n-1)(n-2)+3(n-1)+1.
\end{equation*}%
Thus, we have%
\begin{eqnarray*}
&&\sum\limits_{n=2}^{\infty }n(2n-\cos \alpha -\beta )\frac{(-c/4)^{n-1}}{%
(\kappa )_{n-1}(n-1)!} \\
&=&2\sum\limits_{n=2}^{\infty }(n-1)(n-2)\frac{(-c/4)^{n-1}}{(\kappa
)_{n-1}(n-1)!}+(6-\cos \alpha -\beta )\sum\limits_{n=2}^{\infty }(n-1)\frac{%
(-c/4)^{n-1}}{(\kappa )_{n-1}(n-1)!} \\
&&+(2-\cos \alpha -\beta )\sum\limits_{n=2}^{\infty }\frac{(-c/4)^{n-1}}{%
(\kappa )_{n-1}(n-1)!}. \\
&=&2\sum\limits_{n=3}^{\infty }\frac{(-c/4)^{n-1}}{(\kappa )_{n-1}(n-3)!}%
+(6-\cos \alpha -\beta )\sum\limits_{n=2}^{\infty }\frac{(-c/4)^{n-1}}{%
(\kappa )_{n-1}(n-2)!} \\
&&+(2-\cos \alpha -\beta )\sum\limits_{n=2}^{\infty }\frac{(-c/4)^{n-1}}{%
(\kappa )_{n-1}(n-1)!}. \\
&=&\frac{2(-c/4)^{2}}{\kappa (\kappa +1)}\sum\limits_{n=0}^{\infty }\frac{%
(-c/4)^{n}}{(\kappa +2)_{n}n!}+(6-\cos \alpha -\beta )\frac{(-c/4)}{\kappa }%
\sum\limits_{n=0}^{\infty }\frac{(-c/4)^{n}}{(\kappa +1)_{n}n!} \\
&&+(2-\cos \alpha -\beta )\sum\limits_{n=0}^{\infty }\frac{(-c/4)^{n+1}}{%
(\kappa )_{n+1}(n+1)!} \\
&=&2u_{p}^{\prime \prime }(1)+(6-\cos \alpha -\beta )u_{p}^{\prime
}(1)+(2-\cos \alpha -\beta )(u_{p}(1)-1).
\end{eqnarray*}%
But this last expression is bounded above by $\cos \alpha -\beta ~$if ~(\ref%
{gh}) holds.
\end{proof}

\bigskip By using a similar method as in the proof of Corollary \ref{cor1},
we have the following result.

\begin{corollary}
If $c<0,$ $\kappa >0(\kappa \neq 0,-1,-2,\ldots ),$\textit{\ }then $%
z(2-u_{p}(z))$ is in $\mathcal{UCSPT}(\alpha ,\beta )$ if and only if the
condition (\ref{gh}) is satisfied.
\end{corollary}

The proof of Theorem \ref{th55} (below) is much akin to that of Theorem \ref%
{th66}, and so the details may be omitted.

\begin{theorem}
\label{th55}If $c<0,$ $\kappa >0(\kappa \neq 0,-1,-2,\ldots ),$\textit{\ }%
then $z(2-u_{p}(z))$ is in $\mathcal{UCSPT}(\alpha ,\beta )$ if and only if%
\begin{equation}
e^{(\frac{-c}{4\kappa })}[\allowbreak \frac{c^{2}}{8\kappa }+(6-\cos \alpha
-\beta )(\frac{-c}{4\kappa })+(2-\cos \alpha -\beta )(1-e^{(\frac{c}{4\kappa 
})})]\leq \cos \alpha -\beta .  \label{66}
\end{equation}
\end{theorem}

\section{Inclusion Properties\ }

Making use of Lemma \ref{lem3} , we will study the action of the Bessel
function on the class $\mathcal{UCSPT}(\alpha ,\beta ).$

\begin{theorem}
\label{th2}Let $c<0,$ $\kappa >0(\kappa \neq 0,-1,-2,\ldots ).$ If $\ f\in 
\mathcal{R}^{\tau }(A,B),\ $then $\mathcal{I}(\kappa ,c)f$ is in $\mathcal{%
UCSPT}(\alpha ,\beta )$ if and only if 
\begin{equation}
(A-B)\left\vert \tau \right\vert [2u_{p}^{\prime }(1)+(2-\cos \alpha -\beta
)(u_{p}(1)-1)]\leq \cos \alpha -\beta .  \label{d3}
\end{equation}
\end{theorem}

\begin{proof}
In view of (\ref{b1}), it suffices to show that%
\begin{equation*}
\sum\limits_{n=2}^{\infty }n(2n-\cos \alpha -\beta )\frac{(-c/4)^{n-1}}{%
(\kappa )_{n-1}(n-1)!}\left\vert a_{n}\right\vert \leq \cos \alpha -\beta .
\end{equation*}

Since $f\in \mathcal{R}^{\tau }(A,B),$ then by Lemma \ref{lem3}, we get%
\begin{equation}
\left\vert a_{n}\right\vert \leq \frac{(A-B)\left\vert \tau \right\vert }{n}.
\label{vv}
\end{equation}

Thus, we must show that%
\begin{eqnarray*}
&&\sum\limits_{n=2}^{\infty }n(2n-\cos \alpha -\beta )\frac{(-c/4)^{n-1}}{%
(\kappa )_{n-1}(n-1)!}\left\vert a_{n}\right\vert \\
&\leq &(A-B)\left\vert \tau \right\vert \left[ \sum\limits_{n=2}^{\infty
}(2n-\cos \alpha -\beta )\frac{(-c/4)^{n-1}}{(\kappa )_{n-1}(n-1)!}\right] .
\end{eqnarray*}

The remaining part of the proof of Theorem \ref{th2} is similar to that of
Theorem \ref{th1}, and so we omit the details.
\end{proof}

\section{An integral operator}

In this section, we obtain the necessary and sufficient conditions for the
integral operator $\mathcal{G(}\kappa ,c,z\mathcal{)}$ defined by 
\begin{equation}
\mathcal{G(}\kappa ,c,z\mathcal{)=}\int_{0}^{z}(2-u_{p}(t))dt  \label{pl}
\end{equation}%
to be in $\mathcal{UCSPT}(\alpha ,\beta ).$

\begin{theorem}
\label{th3}If $c<0,$ $\kappa >0(\kappa \neq 0,-1,-2,\ldots )$, then the
integral operator $\mathcal{G(}\kappa ,c,z\mathcal{)}$ is in $\mathcal{UCSPT}%
(\alpha ,\beta )$ if and only if ~the condition (\ref{hh}) is satisfied.
\end{theorem}

\begin{proof}
Since%
\begin{equation*}
\mathcal{G(}\kappa ,c,z\mathcal{)}=z-\sum\limits_{n=2}^{\infty }\frac{%
(-c/4)^{n-1}}{(\kappa )_{n-1}}\frac{z^{n}}{n!}
\end{equation*}%
then in view of (\ref{b1}), we need only to show that%
\begin{equation*}
\sum\limits_{n=2}^{\infty }n(2n-\cos \alpha -\beta )\frac{(-c/4)^{n-1}}{%
(\kappa )_{n-1}n!}\leq \cos \alpha -\beta
\end{equation*}

or, equivalently%
\begin{equation*}
\sum\limits_{n=2}^{\infty }(2n-\cos \alpha -\beta )\frac{(-c/4)^{n-1}}{%
(\kappa )_{n-1}(n-1)!}\leq \cos \alpha -\beta .
\end{equation*}

The remaining part of the proof is similar to that of Theorem \ref{th1}, and
so we omit the details.
\end{proof}

\bigskip The proof of Theorem \ref{th44} and Theorem \ref{th99} are much
akin to that of Theorem \ref{th66}, and so the details may be omitted.

\begin{theorem}
\label{th44}Let $c<0,$ $\kappa >0(\kappa \neq 0,-1,-2,\ldots ).$ If $\ f\in 
\mathcal{R}^{\tau }(A,B),\ $then $\mathcal{I}(\kappa ,c)f$ is in $\mathcal{%
UCSPT}(\alpha ,\beta )$ if and only if 
\begin{equation}
(A-B)\left\vert \tau \right\vert e^{(\frac{-c}{4\kappa })}[\frac{-c}{2\kappa 
}+(2-\cos \alpha -\beta )(1-e^{(\frac{c}{4\kappa })})]\leq \cos \alpha
-\beta .
\end{equation}
\end{theorem}

\begin{theorem}
\label{th99}If $c<0,$ $\kappa >0(\kappa \neq 0,-1,-2,\ldots )$, then the
integral operator $\mathcal{G(}\kappa ,c,z\mathcal{)}$ is in $\mathcal{UCSPT}%
(\alpha ,\beta )$ if and only if the condition (\ref{66}) is satisfied.~
\end{theorem}

\section{Corollaries and consequences}

In this section, we apply our main results in order to deduce each of the
following corollaries and consequences.

\begin{corollary}
If $c<0,$ $\kappa >0(\kappa \neq 0,-1,-2,\ldots ),$then $zu_{p}$ is in $%
\mathcal{SP}_{p}(\alpha )$ if 
\begin{equation}
2u_{p}^{\prime }(1)+(2-\cos \alpha )(u_{p}(1)-1)\leq \cos \alpha .  \label{1}
\end{equation}
\end{corollary}

\begin{corollary}
If $c<0,$ $\kappa >0(\kappa \neq 0,-1,-2,\ldots ),$ then $z(2-u_{p}(z))$ is
in $\mathcal{SP}_{p}\mathcal{T}(\alpha )$ if and only if the condition (\ref%
{1}) is satisfied.
\end{corollary}

\begin{corollary}
If $c<0,$ $\kappa >0(\kappa \neq 0,-1,-2,\ldots ),$ then $zu_{p}$ is in $%
\mathcal{SP}_{p}(\alpha )$ if%
\begin{equation}
e^{(\frac{-c}{4\kappa })}[\frac{-c}{2\kappa }+(2-\cos \alpha )(1-e^{(\frac{c%
}{4\kappa })})]\leq \cos \alpha .  \label{2}
\end{equation}
\end{corollary}

\begin{corollary}
If $c<0,$ $\kappa >0(\kappa \neq 0,-1,-2,\ldots ),$ then $z(2-u_{p}(z))$ is
in $\mathcal{SP}_{p}\mathcal{T}(\alpha )$ if and only if the condition (\ref%
{2}) is satisfied.
\end{corollary}

\begin{corollary}
If $c<0,$ $\kappa >0(\kappa \neq 0,-1,-2,\ldots ),$ then $zu_{p}(z)$ is in $%
\mathcal{UCSP}(\alpha )$ if 
\begin{equation}
2u_{p}^{\prime \prime }(1)+(6-\cos \alpha )u_{p}^{\prime }(1)+(2-\cos \alpha
)(u_{p}(1)-1)\leq \cos \alpha .  \label{3}
\end{equation}
\end{corollary}

\begin{corollary}
If $c<0,$ $\kappa >0(\kappa \neq 0,-1,-2,\ldots ),$\textit{\ }then $%
z(2-u_{p}(z))$ is in $\mathcal{UCSPT}(\alpha )$ if and only if%
\begin{equation}
e^{(\frac{-c}{4\kappa })}[\allowbreak \frac{c^{2}}{8\kappa }+(6-\cos \alpha
)(\frac{-c}{4\kappa })+(2-\cos \alpha )(1-e^{(\frac{c}{4\kappa })})]\leq
\cos \alpha .  \label{4}
\end{equation}
\end{corollary}

\begin{corollary}
Let $c<0,$ $\kappa >0(\kappa \neq 0,-1,-2,\ldots ).$ If $\ f\in \mathcal{R}%
^{\tau }(A,B),\ $then $\mathcal{I}(\kappa ,c)f$ is in $\mathcal{UCSPT}%
(\alpha )$ if and only if 
\begin{equation}
(A-B)\left\vert \tau \right\vert [2u_{p}^{\prime }(1)+(2-\cos \alpha
)(u_{p}(1)-1)]\leq \cos \alpha .
\end{equation}
\end{corollary}

\begin{corollary}
If $c<0,$ $\kappa >0(\kappa \neq 0,-1,-2,\ldots )$, then the integral
operator $\mathcal{G(}\kappa ,c,z\mathcal{)}$ is in $\mathcal{UCSPT}(\alpha
) $ if and only if ~the condition (\ref{1}) is satisfied.
\end{corollary}

\begin{corollary}
Let $c<0,$ $\kappa >0(\kappa \neq 0,-1,-2,\ldots ).$ If $\ f\in \mathcal{R}%
^{\tau }(A,B),\ $then $\mathcal{I}(\kappa ,c)f$ is in $\mathcal{UCSPT}%
(\alpha )$ if and only if 
\begin{equation}
(A-B)\left\vert \tau \right\vert e^{(\frac{-c}{4\kappa })}[\frac{-c}{2\kappa 
}+(2-\cos \alpha )(1-e^{(\frac{c}{4\kappa })})]\leq \cos \alpha .
\end{equation}
\end{corollary}

\begin{corollary}
If $c<0,$ $\kappa >0(\kappa \neq 0,-1,-2,\ldots )$, then the integral
operator $\mathcal{G(}\kappa ,c,z\mathcal{)}$ is in $\mathcal{UCSPT}(\alpha
) $ if and only if ~the condition (\ref{4}) is satisfied.
\end{corollary}

\begin{remark}
If we put $\alpha =0$ in Corollary \ref{cor1}, we obtain Theorem 5 in \cite%
{m} for $\lambda =1$ and $\beta =1.$
\end{remark}

\end{document}